\documentclass[11pt,leqno]{amsart}

\makeatletter
\renewcommand\subsection{\@startsection{subsection}{2}%
  \z@{.5\linespacing\@plus.7\linespacing}{-.5em}%
  {\normalfont\scshape}}
\makeatother
\usepackage[a4paper,lmargin=2cm,rmargin=2cm,tmargin=4cm,bmargin=4cm]{geometry}

\usepackage[centertags]{amsmath}
\usepackage{amsfonts}
\usepackage{amssymb}
\usepackage{amsthm}
\usepackage{dsfont}
\usepackage{mathabx}

\numberwithin{equation}{section}

\usepackage[pdftex]{color}
\usepackage[bookmarks=true,hyperindex,pdftex,
colorlinks, citecolor=blue,linkcolor=blue, urlcolor=blue
]{hyperref}

\usepackage{tikz}
\usepackage{tikz-cd}
\usetikzlibrary{decorations.markings}

\usepackage[normalem]{ulem}
\usepackage{bbm}

\newcommand{\N}{\mathbb{N}}

\newcommand{\1}{\mathds{1}}

\DeclareMathOperator{\diam}{diam\,}
\DeclareMathOperator{\co}{co}

\DeclareMathOperator{\conv}{co}
\DeclareMathOperator{\cconv}{\overline{\conv}}

\DeclareMathOperator*{\esssup}{ess\,sup}

\newcommand{\norm}[1]{\left\Vert#1\right\Vert}
\newcommand{\nnorm}[1]{\left\vvvert#1\right\vvvert}
\newcommand{\abs}[1]{\left\vert#1\right\vert}
\newcommand{\eps}{\varepsilon}

\renewcommand{\geq}{\geqslant}
\renewcommand{\leq}{\leqslant}

\theoremstyle{plain}
\newtheorem{thm}{Theorem}[section]
\newtheorem{cor}[thm]{Corollary}
\newtheorem{lem}[thm]{Lemma}
\newtheorem{prop}[thm]{Proposition}
\newtheorem{theorem}[thm]{Theorem}

\newtheorem{proposition}[thm]{Proposition}

\theoremstyle{definition}

\newtheorem{rem}[thm]{Remark}

\newtheorem{definition}[thm]{Definition}

\author {Christian Cobollo}

\author{Daniel Isert}

\author {Gin\'es L\'opez-P\'erez}

\author {Miguel Mart\'in}

\author {Yo\"el Perreau}

\author {Alicia Quero}

\author {Andr\'es Quilis}

\author {Daniel L. Rodr\'iguez-Vidanes}

\author {Abraham Rueda Zoca}

\address[Cobollo] {Universitat Polit\`ecnica de Val\`encia. Instituto Universitario de Matem\'atica Pura y Aplicada, Camino de Vera, s/n 46022 Valencia, Spain
 \newline
\href{https://orcid.org/0000-0002-5901-5798}{ORCID: \texttt{0000-0002-5901-5798} } }
\email{\texttt{chcogo@upv.es}}

\address[Isert]{Univeristat de Val\`encia. Departamento de an\'alisis matem\'atico, Calle Doctor Moliner, 50, 46100 Burjassot (Valencia), Spain
 \newline
\href{https://orcid.org/0000-0002-4006-2500}{ORCID: \texttt{0000-0002-4006-2500}}}\email{\texttt{isada@alumni.uv.es}}

\address[L\'opez-P\'erez, Mart\'in, and Rueda Zoca] {Universidad de Granada, Facultad de Ciencias.
Departamento de An\'{a}lisis Matem\'{a}tico, 18071 Granada, Spain
 \newline
\href{https://orcid.org/0000-0002-3689-1365}{ORCID: \texttt{0000-0002-3689-1365} }
 \newline
\href{https://orcid.org/0000-0003-4502-798X}{ORCID: \texttt{0000-0003-4502-798X} }
 \newline
\href{https://orcid.org/0000-0003-0718-1353}{ORCID: \texttt{0000-0003-0718-1353} } }
\email{\texttt{glopezp@ugr.es, mmartins@ugr.es,
abrahamrueda@ugr.es}}
\address[Perreau]{University of Tartu, Institute of Mathematics and Statistics, Narva mnt 18, 51009 Tartu linn, Estonia \newline
	\href{https://orcid.org/0000-0002-2609-5509}{ORCID: \texttt{0000-0002-2609-5509}}}
 \email{\texttt{yoel.perreau@ut.ee}}

\address[Quero]{Universidad de Granada, Facultad de Ciencias,
Departamento de An\'{a}lisis Matem\'{a}tico, 18071 Granada, Spain.
\newline
Current address: Czech Technical University in Prague, Faculty of Information Technology, Th\'{a}kurova 9, 160 00, Prague 6, Czech Republic
\newline
\href{http://orcid.org/0000-0003-4534-8097}{ORCID: \texttt{0000-0003-4534-8097} } }
\email{aliciaquero@ugr.es}
 
\address[Quilis] {Universit\'{e} Franche-Comt\'{e}, Laboratoire de math\'{e}matiques de Besançon, UMR CNRS 6623, 16, route de Gray, 25000 Besançon, France
 \newline
\href{https://orcid.org/0000-0001-6022-9286}{ORCID: \texttt{0000-0001-6022-9286} } }
\email{\texttt{andresqsa@gmail.com}}

\address[Rodríguez-Vidanes]{Instituto de Matem\'{a}tica Interdisciplinar (IMI). Universidad Complutense de Madrid, Facultad de Ciencias Matemáticas, Departamento de An\'{a}lisis Matem\'{a}tico y Matem\'{a}tica Aplicada, Plaza de Ciencias 3, 28040 Madrid, Spain 
 \newline
\href{https://orcid.org/0000-0002-1016-096X}{ORCID: \texttt{0000-0002-1016-096X} } }
\email{\texttt{dl.rodriguez.vidanes@ucm.es}}

\title{Banach spaces with small weakly open subsets of the unit ball and massive sets of Daugavet and $\Delta$-points}

\begin{document}



\maketitle

\begin{abstract}
We prove that there exists an equivalent norm $\nnorm{\cdot}$ on $L_\infty[0,1]$ with the following properties:
\begin{enumerate}
    \item The unit ball of $(L_\infty[0,1],\nnorm{\cdot})$ contains non-empty relatively weakly open subsets of arbitrarily small diameter;
    \item The set of Daugavet points of the unit ball of $(L_\infty[0,1],\nnorm{\cdot})$ is weakly dense;
    \item The set of ccw $\Delta$-points of the unit ball of $(L_\infty[0,1],\nnorm{\cdot})$ is norming.
\end{enumerate}
We also show that there are points of the unit ball of $(L_\infty[0,1],\nnorm{\cdot})$ which are not $\Delta$-points, meaning that the space $(L_\infty[0,1],\nnorm{\cdot})$ fails the diametral local diameter 2 property. Finally, we observe that the space $(L_\infty[0,1],\nnorm{\cdot})$ provides both alternative and new examples that illustrate the differences between the various diametral notions for points of the unit ball of Banach spaces.
\end{abstract}

\markboth{Cobollo, Isert, L\'opez-P\'erez, Mart\'in, Perreau, Quero, Quilis, Rodr\'iguez and Rueda Zoca}{Banach spaces with small weakly open subsets of the unit ball and massive sets of Daugavet and $\Delta$-points}

\section{Introduction}\label{sec:introduction}

Recall that a Banach space $X$ is said to have the \textit{Daugavet property} if every rank one bounded operator $T:X\longrightarrow X$ satisfies the \textit{Daugavet equation}
\begin{equation}\label{DE}\tag{DE}
\Vert I+T\Vert=1+\Vert T\Vert,
\end{equation}
where $I:X\longrightarrow X$ stands for the identity operator. Furthermore, if $X$ has the Daugavet property, then every weakly compact operator $T:X\to X$ satisfies (\ref{DE}). Since the Daugavet equation is a stress of the operator norm's triangle inequality, it is natural to expect that it will impose severe restrictions on the underlying operator. As a matter of fact, if $\Vert T\Vert$ is an eigenvalue of $T$, then $T$ satisfies the Daugavet equation, and the converse holds true if the space $X$ is uniformly convex \cite[Lemma 11.3 and Theorem 11.10]{abal}. 

Actually, the Daugavet property puts very strong constraints on the structure of the underlying Banach space. An old result in this line is that a Banach space with the Daugavet property cannot be linearly embedded into any Banach space with an unconditional basis (see e.g. \cite[Theorem 3.2]{wernersurvey01}). Further restrictions follow from the celebrated geometric characterisation of the Daugavet property exhibited in \cite[Lemma 2.1]{kssw} stated as follows:
a Banach space $X$ has the Daugavet property if and only if every point $x\in S_X$ satisfies the following condition: given any slice $S$ of $B_X$ and any $\varepsilon>0$, there exists $y\in S$ such that
$$\Vert x-y\Vert>2-\varepsilon.$$
The latter characterisation, which still holds with respect to non-empty relatively weakly open subsets (resp. convex combinations of slices) \cite[Lemma 3]{shv00}, shows that spaces with the Daugavet property live in the universe of Banach spaces far away from Asplundness and Radon-Nikodym property. Indeed, the above characterisation allows to prove that if $X$ has the Daugavet property, then $X$ contains an isomorphic copy of $\ell_1$, and every slice, relatively weakly open subset and convex combination of slices of $B_X$ has diameter two.

Very recently, local versions of the Daugavet property have been considered in the following sense.

\begin{definition}
Let $X$ be a Banach space and let $x\in S_X$. We say that $x$ is
\begin{enumerate}
\item a \textit{Daugavet point} if, for every slice $S$ of $B_X$ and every $\varepsilon>0$, there exists $y\in S$ such that $\Vert y-x\Vert>2-\varepsilon$,
\item a \textit{super Daugavet point} if, for every non-empty relatively weakly open subset $W$ of $B_X$ and every $\varepsilon>0$, there exists $y\in W$ such that $\Vert y-x\Vert>2-\varepsilon$,
\item a \textit{ccs Daugavet point} if, for every convex combination of slices $C$ of $B_X$ and every $\varepsilon>0$, there exists $y\in C$ such that $\Vert y-x\Vert>2-\varepsilon$.
\end{enumerate}
\end{definition}

A classical result, often known as Bourgain's lemma, establishes that every non-empty relatively weakly open subset of $B_X$ contains a convex combination of slices of $B_X$ (see e.g. \cite[Lemma II.1]{ggms}). As an immediate consequence we infer that every ccs Daugavet point is a ``ccw Daugavet point'', meaning that the property of the definition actually holds for every convex combination of non-empty relatively weakly open subsets of $B_X$. In particular, every ccs Daugavet point is a super Daugavet point. Furthermore, it is known that the mere existence of a ccs Daugavet point implies that every convex combination of slices (and of weak open subsets) of the unit ball of the underlying space has diameter 2 \cite[Proposition 3.12]{mpr23}. Apart from finite dimensional considerations, this is surprisingly the only known isomorphic obstruction to the existence of diametral points, see below for more details. 

Variants of the above notions restricting to slices, weakly open subsets, and convex combinations of slices and weakly open subsets containing a given point, were also considered.

\begin{definition}
Let $X$ be a Banach space and let $x\in S_X$. We say that $x$ is
\begin{enumerate}
\item a \textit{$\Delta$-point} if, for every slice $S$ of $B_X$  with $x\in S$ and every $\varepsilon>0$, there exists $y\in S$ such that $\Vert y-x\Vert>2-\varepsilon$,
\item a \textit{super $\Delta$-point} if, for every non-empty relatively weakly open subset $W$ of $B_X$ with $x\in W$ and every $\varepsilon>0$, there exists $y\in W$ such that $\Vert y-x\Vert>2-\varepsilon$,
\item a \textit{ccs $\Delta$-point} if, for every convex combination of slices $C$ of $B_X$ with $x\in C$ and every $\varepsilon>0$, there exists $y\in C$ such that $\Vert y-x\Vert>2-\varepsilon$;
\item a \textit{ccw $\Delta$-point} if, for every convex combination of non-empty relatively weakly open subsets $D$ of $B_X$ with $x\in D$ and every $\varepsilon>0$, there exists $y\in D$ such that $\Vert y-x\Vert>2-\varepsilon$.
\end{enumerate}
\end{definition}

The notions of Daugavet and $\Delta$-points were introduced in \cite[Section 1]{ahlp20}, whereas the rest of notions go back to \cite[Definitions 2.4 and 2.5]{mpr23}. See \cite{aalmppv23,aalmppvfirst23,jr22,klt23,mpr23,veeorg} for further research on these notions. In particular, note that it is still unknown whether every ccs $\Delta$-point has to be a super $\Delta$-point, and whether the notions of ccs and ccw $\Delta$-points are different. This is due to the subtle failure of a localization of Bourgain's lemma (see e.g. \cite[Remark 2.3]{mpr23}). However, all the other notions are known to be different and can even present extreme differences, see \cite{mpr23} for more details.  

In view of the fact that the Daugavet property imposes strong restrictions on the geometric structure of the given space, a natural question is how the mere presence of Daugavet or $\Delta$-points affect the geometric structure of the underlying Banach space. Although it was proved in \cite{almp22} that finite dimensional spaces contain no $\Delta$-points and that the notion strongly negates some isometric properties of Banach spaces (asymptotic uniform smoothness and weak$^*$ asymptotic uniform convexity \cite{aalmppvfirst23, almp22}, or existence of subsymmetric bases \cite{almt20} as well as unconditional bases with small constants of unconditionality which are either shrinking or boundedly complete \cite{aalmppvfirst23}), surprising examples have recently shown the purely isometric nature of these local notions. To name a few, there exists a space with a 1-unconditional basis and a weakly dense subset of Daugavet point \cite{almt20}, there exists a Lipschitz-free space with the RNP and a Daugavet point which is both isomorphic to $\ell_1$ and isometric to a dual space \cite{aalmppvfirst23, veeorg}, and there exists an equivalent norm on $\ell_2$ for which the unit vector basis $e_1$ is simultaneously a super Daugavet point and a ccw $\Delta$-point \cite{hlpv23}. Actually, every infinite dimensional Banach space can be renormed with a $\Delta$-point \cite{aalmppvfirst23}, and every Banach space with a weakly-null unconditional Schauder basis can be renormed with a super Daugavet point \cite{hlpv23}. 

The various $\Delta$-notions can be seen as extreme opposites to the classical notions of denting points, points of continuity and points of strong regularity (also see \cite{cj23} for precise quantitative formulations of this statement). They are localized versions of the so called ``diametral diameter 2 properties'' (\emph{DLD2P}, \emph{DD2P} and \emph{DSD2P}) that have previously appeared in the literature under various names, but were formally introduced in \cite{blr18}. With this terminology, the DLD2P (resp. DD2P) asks all the elements of the unit sphere of a Banach space to be $\Delta$-points (resp.  super $\Delta$-points). The DSD2P was originally defined by asking all the points inside of the unit ball of a Banach space to be ccs $\Delta$-points, but it turned out to be equivalent to the Daugavet property \cite{kadets20}. On the other hand, its restricted version (the \emph{restricted DSD2P} \cite{mpr23}), as well as the DLD2P and the DD2P, are known to be strictly weaker properties. Although the Daugavet property can be characterized by Daugavet, super Daugavet or ccs Daugavet points, it is currently unknown whether the three remaining diametral properties are equivalent. Furthermore, it is unknown whether the DLD2P forces all the weakly open subsets of the unit ball to have diameter 2 (but note that there exists a space with the DD2P, the restricted DSD2P and convex combinations of slices of arbitrarily small diameter in its unit ball \cite{ahntt}). 

The example from \cite{almt20} provides an interesting insight to this question. Indeed, the space that was constructed there with a weakly dense subset of Daugavet points and a 1-unconditional basis admits non-empty relatively weakly open subsets of arbitrarily small diameter in its unit ball. In fact, each of the Daugavet points in the considered weakly dense set is a point of continuity for the identity mapping $I:(B_X, w)\to (B_X,\norm{\cdot})$ (in other words, it has relative weak open neighborhoods of arbitrarily small diameter). However, this space cannot contain any point satisfying a stronger diametral condition, as it was proved in \cite{almt20} (resp. \cite{mpr23}) that spaces with a 1-unconditional basis contain neither super $\Delta$-points nor ccs $\Delta$-points. Thus, at this point, a natural question is how big the set of stronger notions than Daugavet  and $\Delta$-points can be in a Banach space where there are non-empty relatively weakly open subsets of arbitrarily small diameter. 

In view of this fact, during the last week of June 2023, in the framework of 2023 ICMAT-IMAG Doc-Course in Functional Analysis, a supervised research program was celebrated at IMAG (Granada), where we considered the following question: How massive can the sets of Daugavet, super $\Delta$, super Daugavet and ccs/ccw $\Delta$-points be in a Banach space having non-empty relatively weakly open subsets of arbitrary small diameter in its unit ball? The main goal of the project was to study the renorming techniques from \cite{blrjmaa}, where it is proved that every Banach space containing $c_0$ can be renormed in such a way that all the slices of the new unit ball have diameter 2, whereas it admits weakly open subsets of arbitrarily small diameter, and to try to build a similar renorming in a more suitable context for our study, namely in the space $L_\infty[0,1]$. The idea is also inspired by the construction from \cite[Section~4.6]{mpr23}, where similar techniques were used in order to produce an example of a super Daugavet point that is not a ccs $\Delta$-point. 

The main aim of the present paper is to present the results obtained in this workshop. We prove that the space $L_\infty[0,1]$ admits an equivalent renorming such that the new unit ball contains non-empty relatively weakly open subsets of arbitrarily small diameter and such that the sets of Daugavet points and super $\Delta$-points are as big as they can be taking into account that its unit ball contains non-empty weakly open subsets of small diameter. This is a big difference with the above mentioned example of \cite{almt20}, where the set of super $\Delta$-points is empty. Furthermore, we show that this space also contains points which are simultaneously super Daugavet and ccw $\Delta$, which is the strongest diametral notion we can get in this context. We collect the results in the following theorem.

\begin{theorem}\label{thm:main_theorem}
For every $\eps\in(0,1)$, there exists an equivalent norm $\nnorm{\cdot}_\eps$ on $L_\infty[0,1]$ with the following properties:
\begin{enumerate}
    \item For every $f\in L_\infty[0,1]$, $\norm{f}_\infty\leq \nnorm{f}_{\eps}\leq\frac{1}{1-\eps}\norm{f}_\infty$;
    \item The unit ball of $(L_\infty[0,1],\nnorm{\cdot}_\eps)$ contains non-empty relatively weakly open subsets of arbitrarily small diameter;
    \item The set of Daugavet points of the unit ball of $(L_\infty[0,1],\nnorm{\cdot}_\eps)$ is weakly dense;
    \item The set of ccw $\Delta$-points of the unit ball of $(L_\infty[0,1],\nnorm{\cdot}_\eps)$ is norming (in other words, every slice of the unit ball contains a ccw $\Delta$-point); 
    \item There are points of the unit ball of $(L_\infty[0,1],\nnorm{\cdot}_\eps)$ which are: \begin{enumerate}
        \item Simultaneously super Daugavet points and ccw $\Delta$-points;
        \item Simultaneously  Daugavet points and preserved extreme points (hence also ccw $\Delta$-points), but not super Daugavet points;
        \item Simultaneously Daugavet points and points of continuity. 
    \end{enumerate}
\end{enumerate} Furthermore, if $\eps$ is smaller than $1/7$, then there are points of the unit ball of $(L_\infty[0,1],\nnorm{\cdot}_\eps)$ which are not $\Delta$-points (in other words, $(L_\infty[0,1],\nnorm{\cdot}_\eps)$ fails the DLD2P).
\end{theorem}

In particular, in the above renorming there are Daugavet points which are not super $\Delta$-points and there are ccw $\Delta$-points which are not super Daugavet points. Even though it was already known that these notions are not equivalent (see \cite{mpr23} for references), the various counterexamples from the literature were obtained with different techniques. Theorem \ref{thm:main_theorem} shows that such counterexamples may live in the same Banach space. Furthermore, it is, to our knowledge, the first example of a Banach space which contains points that are both Daugavet and ccw $\Delta$, but not super Daugavet.

\section{Notation and preliminary results}\label{sec:preliminaries}

Given a Banach space $X$, $B_X$ (resp. $S_X$) stands for the closed unit ball (resp. the unit sphere) of $X$. We denote by $X^*$ the topological dual of $X$. By a slice of $B_X$, we mean any non-empty subset of $B_X$ given as the intersection of $B_X$ with an open half-space. Every slice $S$ of $B_X$ can be written as $S=S(B_X,f,\delta)$, where $f$ is a norm one functional on $X$, $\delta>0$ and\begin{equation*}
S(B_X,f,\delta):=\{x\in B_X\colon f(x)>1-\delta\}.
\end{equation*} If $A$ is a subset of a Banach space $X$, we denote by $\conv{A}$ (resp. $\cconv{A}$) the convex hull (resp. the closure of the convex hull) of $A$. Recall that a subset $A$ in the unit ball of Banach space $X$ is said to be \emph{norming} if $\norm{x^*}=\sup_{x\in A}{\abs{x^*(x)}}$ for every $x^*\in X^*$. In particular, if $A$ is a symmetric subset of $B_X$, then this property is equivalent to $A$ satisfying $B_X=\cconv{A}$ (in other words, to every slice of $B_X$ containing an element of $A$).  We deal with real Banach spaces only.

Let $\mu$ be the Lebesgue measure on the segment $[0,1]$. Recall that two measurable subsets $A$ and $B$ of $[0,1]$ are said to be \emph{essentially disjoint} if $\mu(A\cap B)=0$. The space $L_\infty[0,1]$ stands for the classical Banach space of all equivalent classes of $\mu$-essentially bounded functions on $[0,1]$ equipped with the norm given by the essential supremum. Recall that the following criteria provides a practical way of testing whether a given sequence in $L_\infty[0,1]$ is weakly-null (see e.g. \cite[Theorem~8.7]{toland20}).

\begin{thm}\label{thm:weakly_null_sequences_L_infty}
 A bounded sequence $(u_n)$ in $L_\infty[0,1]$ converges weakly to $0$ if and only we can find, for every $\delta>0$ and $(k_j)$ increasing sequence of natural numbers, some $J\in\N$ such that \begin{equation*}
 \mu\left(\bigcap_{j=1}^J\left\{t\in[0,1]:\ \abs{u_{k_j}(t)}>\delta\right\}\right)=0.
\end{equation*} In particular, every bounded sequence of functions with pairwise essentially disjoint supports in $L_\infty[0,1]$ is weakly-null.
\end{thm}

We then recall some classical definitions from Banach space geometry. Given a convex set $A$ in a vector space $X$, a point $x_0\in A$ is said to be \emph{extreme} if the condition $x_0=\frac{y+z}{2}$ for $y,z\in A$ forces $y=z=x_0$. Given a bounded, closed, and convex subset $C$ of a Banach space $X$, a point $x_0\in C$ is a \textit{preserved extreme point} if $x_0$ is an extreme point in $\overline{C}^{w^*}$, where the closure is taken in the $w^*$-topology of $X^{**}$. For easy reference, let us point out the following characterisation of preserved extreme points (which proof can be found, for instance, in \cite[Proposition 0.1.3]{luiscathesis}). 

\begin{proposition}\label{prop:carapreext}
Let $X$ be a Banach space and let $C\subseteq X$ be a bounded, closed, and convex set. Let $x_0\in C$. The following are equivalent:
\begin{enumerate}
    \item $x_0$ is a preserved extreme point of $C$;
    \item The slices of $C$ containing $x_0$ form a neighbourhood basis of $x_0$ in $C$ for the weak topology;
    \item For every pair of nets $(y_s)$ and $(z_s)$ in $C$ such that $\frac{y_s+z_s}{2}\rightarrow x_0$ weakly, we have $y_s\rightarrow x_0$ weakly.
\end{enumerate}
\end{proposition}

Given a Banach space $X$ and a subset $A\subseteq X$, a point $x_0\in A$ is said to be a \textit{point of continuity} if, for every $\varepsilon>0$, there exists a weakly open subset $W\subseteq A$ with $x_0\in W$ and $\diam(W)<\varepsilon$. Observe that this means that the identity mapping $I:(A,w)\longrightarrow (A,\Vert\cdot\Vert)$ is continuous at $x_0$. In turn, this is equivalent to the fact that if a net $(x_s)$ of elements of $A$ satisfies that $x_s\overset{w}{\rightarrow} x_0$, then $\Vert x_s-x_0\Vert\rightarrow 0$. A closed and bounded set $B$ (resp. a closed convex and bounded set $C$) in a Banach space $X$ is said to have the \emph{point of continuity property} (resp. \textit{convex point of continuity property (CPCP)}) if every closed subset $A$ of $B$ (resp. every closed and convex subset $A$ of $C$) contains a point of continuity.

We finally recall the definition of the ``Summing Tree Simplex'' from \cite{aorlecture} that was constructed in order to distinguish between the CPCP and the PCP for subsets of Banach spaces. This set will be the stepping stone for our renorming of $L_\infty[0,1]$. Let $\N^{<\omega}$ be the set of all ordered finite sequences of positive integers including the empty sequence denoted by $\emptyset$. If $\alpha=(\alpha_1,\ldots,\alpha_n)\in\N^{<\omega}$, the length of $\alpha$ is $|\alpha|=n$ and $|\emptyset|=0$. For simplicity, we will sometimes identify $\N^1$ with $\N$ and denote by $i$ the sequence $(i)$ with one element $i\in\N$. We use the natural order in $\N^{<\omega}$ given by:
$$
\alpha\preceq\beta \text{ if } |\alpha|\leq|\beta| \text{ and } \alpha_i=\beta_i \text{ for all } i\in\{1,\ldots,|\alpha|\},
$$
and $\emptyset\preceq\alpha$ for any $\alpha\in \N^{<\omega}$. We denote by $\alpha\smallfrown i$ the finite sequence resulting from the concatenation of an element $\alpha\in\N^{<\omega}$ with the sequence $i=(i)$ with only one element $i\in\N$. 

Let $c_0(\N^{<\omega})$ be the completion of the space $c_{00}(\N^{<\omega})$ of all finitely supported families of real numbers indexed by $\N^{<\omega}$ with the supremum norm. Then $c_0(\N^{<\omega})$ is isometric to the usual space $c_0$. From now on, in order to distinguish with the norm from the space $L_\infty[0,1]$, we will denote by $\norm{.}$ the supremum norm on $c_0$ and $c_0(\N^{<\omega})$, and by $\norm{.}_\infty$ the essential supremum norm on $L_\infty[0,1]$.  Let $(e_\alpha)_{\alpha\in \N^{<\omega}}$ be the unit vector basis of $c_{00}(\N^{<\omega})$ and $(e_\alpha^*)_{\alpha\in \N^{<\omega}}$ be the sequence of biorthogonal functionals. For a given  $\alpha\in \N^{<\omega}$, let \begin{equation*}
    x_\alpha:=\sum_{\beta\preceq \alpha}e_\beta.
\end{equation*} We consider the set \begin{equation*}
    K:=\cconv\{x_\alpha\}_{\alpha\in \N^{<\omega}}\subset S_{c_0}^+.
\end{equation*} Some properties of the set $K$ are given in \cite[Theorem~1.1]{aorlecture}.  In particular, it is proved there that $K$ has the CPCP but fails the PCP. We end the present section by providing a few more properties for $K$. 

\begin{lem}
For every $x\in K$ and for every slice $S$ of $K$, $\sup_{y\in S} \norm{x-y}=1$.
\end{lem}

\begin{proof} Observe that for every $z\in \conv\{x_\alpha\}_{\alpha\in \N^{<\omega}}$ and for every $\alpha\in \N^{<\omega}$, we have $\lim\norm{z-x_{\alpha\smallfrown n}}=1$. Thus, since every slice of $K$ contains some $x_\alpha$, and since $x_{\alpha\smallfrown n}\to_n x_\alpha$ weakly, the conclusion follows from an easy density argument.
\end{proof}

From the fact that $K$ has the CPCP it is immediate to infer that $K$ contains non-emtpy relatively weakly open subsets of arbitrarily small diameter. However, we will describe a particular family of non-empty relatively weakly open subsets of small diameter because they will be useful in order to localise ccw $\Delta$-points which are not super Daugavet points in the final renorming of $L_\infty[0,1]$ (see Remark \ref{rem:superdeltanotsuperdauga}).

\begin{lem}\label{lem:small_weakly_open_subsets_K}
For $n\in \N$ and $\rho\in(0,1/n)$, let \begin{equation*}
    V_{n,\rho}:=\bigcap_{i=1}^n\{z\in K\colon \ e_i^*(z)>1/n-\rho\}.
\end{equation*}
Then $V_{n,\rho}$ is a non-empty relatively weakly open subset of $K$ with diameter smaller than $2/n+2n\rho$ .
\end{lem}

\begin{proof}
    For $i\in\{1,\dots, n\}$, let $x_i:=x_{(i)}$. Then $x_0:=\frac{1}{n}\sum_{i=1}^n x_i\in V_{n,\rho}$. Clearly, it is enough to prove that for every $z\in \conv\{x_\alpha\}_{\alpha\in \N^{<\omega}}\cap V_{n,\rho}$, $\norm{x_0-z}\leq 1/n+n\rho$. Fix such a $z$, and write $z=\sum_{l=1}^L\lambda_lx_{\alpha_l}$ with $\lambda_l>0$, $\sum_{l=1}^L\lambda_l=1$, and $\alpha_l\in \N^{<\omega}$. For every $i\in\N$, let \begin{equation*}
        A_i:=\{l\colon (i)\preceq \alpha_l\}.
    \end{equation*}
    Since $z\in V_{n,\rho}$, we have $e_i^*(z)=\sum_{l\in A_i}\lambda_l>1/n-\rho$ for every $i\leq n$. So observe that for any given $j\in\{1,\dots, n\}$, we have \begin{equation*}
     \sum_{l\in A_j}\lambda_l =\sum_{i=1}^n\sum_{l\in A_i}\lambda_l-\sum_{\substack{i=1\\i\neq j}}^n\sum_{l\in A_i}\lambda_l \leq 1-(n-1)(1/n-\rho)=1/n+(n-1)\rho.
    \end{equation*} In the same way, \begin{equation*}
     \sum_{i>n}\sum_{l\in A_i}\lambda_l \leq \sum_{i\in\N}\sum_{l\in A_i}\lambda_l-\sum_{i=1}^n\sum_{l\in A_i}\lambda_l \leq 1-n(1/n-\rho)=n\rho.
    \end{equation*} Now let us define $v=x_0-z=\frac{1}{n}\sum_{i=1}^nx_i-\sum_{l=1}^L\lambda_lx_{\alpha_l}$ and let us fix $\beta\in \N^{<\omega}$. We want to evaluate $\abs{v(\beta)}$. There are three cases to consider. 
    
    \textbf{Case~1.}  If $\beta=\emptyset$, then $v(\beta)=0$, so there is nothing to do.
    
    \textbf{Case~2.} If $\abs{\beta}>1$, take $j_\beta\in \mathbb N$ such that $(j_\beta)\preceq \beta$.
    Then, either there is no $l\in \{1,\ldots,L\}$ such that $(j_\beta)\preceq \alpha_l$ in which case $v(\beta)=0$, or there is an $l\in \{1,\ldots,L\}$ such that $(j_\beta)\preceq \alpha_l$, and $\abs{v(\beta)} = \abs{\sum_{l=1}^L\lambda_lx_{\alpha_l}(\beta)}\leq \sum_{l\in A_{j_\beta}}\lambda_l$. Hence $\abs{v(\beta)}\leq \max\{1/n+(n-1)\rho, n\rho\}\leq 1/n+n\rho$. 
    
    \textbf{Case~3.} If $\beta=(j_\beta)$ for some $j_\beta \in \mathbb N$, then either $j_\beta>n$, and $\abs{v(\beta)}\leq \sum_{l\in A_{j_\beta}}\lambda_l\leq n\rho$, or $j_\beta\leq n$, and $\abs{v(\beta)}=\abs{1/n-\sum_{l\in A_{j_\beta}}\lambda_l}\leq n\rho$ because $1/n-\rho<\sum_{l\in A_{j_\beta}}\lambda_l\leq 1/n+(n-1)\rho$. 
    
    It follows that $\norm{x_0-z}=\sup_{\beta\in \N^{<\omega}}\abs{v(\beta)}\leq 1/n+n\rho$, as we wanted.
\end{proof}

For easy future reference, we end the section with the following lemma, whose proof follows from \cite[P.82]{aorlecture}.

\begin{lem}
    Every $x_\alpha$ is a preserved extreme point of $K$.
\end{lem}

%\begin{proof}
 %   Fix $\alpha\in \N^{<\omega}$ and pick two nets $(u_s)$ and $(v_s)$ in $K$ such that $\frac{u_s+v_s}{2}\to x_\alpha$ weakly. Then \begin{equation*}
%        \frac{u_s(\beta)+v_s(\beta)}{2}=e_\beta^*\left(\frac{u_s+v_s}{2}\right)\to x_\alpha(\beta)=\1_{[\emptyset,\alpha]}(\beta).
 %   \end{equation*} Now since $u_s(\beta)$ and $v_s(\beta)$ are both in $[0,1]$, this implies that $u_s(\beta),v_s(\beta)\to x_\alpha(\beta)$ for every $\beta\in\N^{<\alpha}$. Otherwise, there would be tails for one of the nets where the values stay away from $1$, respectively $0$, but the other term cannot compensate since every element in $K$ is positive. It follows that $u_s,v_s\to x_\alpha$ weakly, and $x_\alpha\in\preext{K}$. 
%\end{proof}

\section{Main result}\label{sec:main_result}

The aim of the section is to prove Theorem~\ref{thm:main_theorem}. Let $\{A_\alpha\}_{\alpha\in \N^{<\omega}}$ be a family of pairwise disjoint non-empty open subsets of $[1/2,1]$. Then, for every $\alpha\in \N^{<\omega}$, let $f_\alpha:=\1_{A_\alpha}$ and let $E_\alpha$ be the norm one functional on $L_\infty[0,1]$ given by \begin{equation*}E_\alpha(f):=1/\mu(A_\alpha)\cdot\int_{A_\alpha}f d\mu
 \end{equation*} for every $f\in L_\infty[0,1]$. We will use the $f_\alpha$ to construct a positive isometric embedding of $c_0(\N^{<\omega})$ into $L_\infty[0,1]$. For every $z:=(z_\alpha)_{\alpha\in \N^{<\omega}}\in c_{00}(\N^{<\omega})$, we define \begin{equation*}
    \Phi(z):= \sum_{\alpha \in \N^{<\omega}} z_\alpha f_\alpha.
\end{equation*} By the disjointness of the support of the functions $f_\alpha$, this map is clearly isometric and sends positive sequences to positive functions. Thus it can be extended by density to a positive isometric embedding $\Phi\colon c_0(\N^{<\omega})\to L_\infty[0,1]$. Furthermore, observe that, by construction, $\Phi$ satisfies the equation \begin{equation}\label{transfer_eqt}
 E_\alpha\circ\Phi=e^*_\alpha,
 \end{equation} that is $E_\alpha(\Phi(z))=z_\alpha$ for every $z:=(z_\alpha)_{\alpha\in \N^{<\omega}}\in c_0(\N^{<\omega})$.
 
Let $K_0:=\Phi(K)$ be the image of the subset $K$ of $c_0(\N^{<\omega})$ from the preliminary section, and fix $\eps\in(0,1)$. We consider the equivalent norm $\nnorm{\cdot}_\eps$ on $L_\infty[0,1]$ given by the Minkowski functional of the set \begin{equation*}
B_\eps:= \cconv\big((2K_0-\1) \cup (-2K_0+\1) \cup \big((1-\varepsilon) B_{L_\infty[0,1]} + \varepsilon B_{\ker(E_0)} \big)\big),
 \end{equation*} where $\1:=\1_{[0,1]}$ and $E_0$ is the norm one functional on $L_\infty[0,1]$ given by $E_0(f):=4\cdot\int_0^{1/4}f d\mu$ for every $f\in L_\infty[0,1]$. Observe that $(1-\eps)B_{L_\infty[0,1]}\subset B_\eps\subset B_{L_\infty[0,1]}$, which means \begin{equation*}
     (1-\eps)\nnorm{\cdot}_\eps\leq \norm{\cdot}_{\infty}\leq\nnorm{\cdot}_\eps.
 \end{equation*} We will prove that this renorming of $L_\infty[0,1]$ satisfies all the properties of Theorem~\ref{thm:main_theorem}.
 
 Let $A:=(2K_0-\1)$, $B:=(1-\varepsilon) B_{L_\infty[0,1]} + \varepsilon B_{\ker(E_0)}$ and $A_\eps:=A\cup -A\cup B$. For every $\alpha\in\N^{<\omega}$, let $h_\alpha:=\sum_{\beta\preceq\alpha}f_\beta$ and $u_\alpha:=2h_\alpha-\1$. Observe that $A=\cconv\{u_\alpha\}_{\alpha\in \N^{<\omega}}$, and that $E_0(\psi)=-1$ for every $\psi\in A$ and $E_0(\varphi)\in [-1+\eps,1-\eps]$ for every $\varphi\in B$. We will start by proving that our renorming produces weakly open subsets of arbitrarily small diameter in the new unit ball. 
 
\begin{prop}\label{prop:small_weakly_open_subsets_B_eps}
 The set $B_\varepsilon$ admits non-empty relatively weakly open subsets of arbitrarily small diameter.
\end{prop}

\begin{proof}
For $n\in \N$ and $\rho >0$, consider 
\begin{equation*}
    \tilde V_{n,\rho}:= \left\{ f \in B_\varepsilon : E_0(f)<-1+\rho\ \text{and}\ E_i(f) > 2\left(\frac{1}{n}-\rho \right)-1  \text{ for every }i\in \{1,...,n\} \right\},
\end{equation*}
where $i$ stands the sequence $(i)$. Note that for $x_0:=\frac{1}{n} \sum_{i=1}^n x_i$, we have that $f_0:=2 \Phi(x_0) -\mathds{1} \in \tilde V_{n,\rho}$. By density, it is enough to  find an upper bound for the distance of $f$ to $f_0$ for every $f\in \tilde V_{n,\rho} \cap \co A_\varepsilon$. So pick such an $f$, and write
\begin{equation*}
    f= \lambda_1 f_1 + \lambda_2 f_2 + \lambda_3 f_3 
\end{equation*}
with $\lambda_1,\lambda_2,\lambda_3\in [0,1]$, $\sum_{i=1}^3\lambda_i=1$, $f_1\in A$, $f_2\in -A$ and $f_3 \in B$.
Then, evaluating against the functional $E_0$, we get \begin{equation*}
    -1+\rho>E_0(f)\geq-\lambda_1+\lambda_2-(1-\eps)\lambda_3 \geq -\lambda_1-(1-\eps)\lambda_3,
\end{equation*} hence \begin{equation*}
    1-\rho<\lambda_1+(1-\eps)\lambda_3.
\end{equation*} In particular, \begin{equation*}
    1-\rho<\lambda_1+\lambda_3=1-\lambda_2,
\end{equation*} and we get $\lambda_2<\rho$. Furthermore, since \begin{equation*}
    \lambda_1=1-\lambda_2-\lambda_3 \leq 1-\lambda_3,
\end{equation*} we have \begin{equation*}
    1-\rho<1-\lambda_3+(1-\eps)\lambda_3,
\end{equation*} and thus $\lambda_3<\rho/\eps$. Finally \begin{equation*}
    \lambda_1=1-\lambda_2-\lambda_3>1-(1+1/\eps)\rho.
\end{equation*} 

It follows that \begin{equation*}
\nnorm{f-f_1}_\eps\leq (1-\lambda_1)\nnorm{f_1}_\eps+\lambda_2\nnorm{f_2}_\eps+\lambda_3\nnorm{f_3}_\eps < 2(1+1/\eps)\rho.
    \end{equation*}
Write $f_1:=2\Phi(z)-\mathds{1}$ with $z\in K$. For every $i\in\{1,\dots, n\}$, we have \begin{equation*}
2(1/n-\rho )-1< E_i(f)=\lambda_1E_i(f_1)+\lambda_2E_i(f_2)+\lambda_3E_i(f_3)<E_i(f_1)+2(1+1/\eps)\rho,
\end{equation*}
thus
\begin{equation*}
    z_i=E_i(\Phi(z))=\frac{E_i(f_1)+1}{2}>\frac{1}{n}-(2+1/\eps)\rho,
\end{equation*}
which means $z\in V_{n,\tilde{\rho}}$ for $\tilde{\rho}=(2+1/\eps)\rho$. 
So by Lemma~\ref{lem:small_weakly_open_subsets_K}, we get \begin{equation*}
    \nnorm{f_1-f_0}_\eps=2\nnorm{\Phi(x_0)-\Phi(z)}_\eps\leq \frac{2}{1-\eps}\norm{\Phi(x_0)-\Phi(z)}_\infty=\frac{2}{1-\eps}\norm{x_0-z}\leq \frac{4/n+4n\tilde{\rho}}{1-\eps}.
\end{equation*} The conclusion follows.
\end{proof}

Actually, we can say a bit more in that regard: the set $B_\eps$ admits points of continuity. Indeed, the latter claim immediately follows from the fact that the set $K$ itself admits points of continuity (it has the CPCP, see \cite[Theorem~1.1(c)]{aorlecture}) together with the following result. 

 \begin{prop}\label{prop:PC_transfer}
    Let $z$ be a point of continuity of $K$. Then $2\Phi(z)-\1$ is a point of continuity of $B_\eps$.
\end{prop}

\begin{proof}
Let $z\in K$ be a point of continuity. To show that $f:=2\Phi(z)-\1$ is a point of continuity of $B_\eps$, it is enough by density to prove that for every net $(f_s)$ in $\conv{A_\eps}$, we have that $f_s\to f$ weakly if and only if $\nnorm{f-f_s}_\eps\to 0$. So consider such a net, and for every $s$, write \begin{equation*}
    f_s= \lambda_s^1 f_s^1 + \lambda_s^2 f_s^2 + \lambda_s^3 f_s^3 
\end{equation*}
with $\lambda_s^1,\lambda_s^2,\lambda_s^3\in [0,1]$, $\sum_{i=1}^3\lambda_s^i=1$, $f_s^1\in A$, $f_s^2\in -A$ and $f_s^3 \in B$. If $f_s\to f$ weakly, then \begin{equation*}
E_0(f_s)\to E_0(f)=-1.
\end{equation*} Now since $E_0(f_s^1)=-1$, $E_0(f_s^2)=1$ and $E_0(f_s^3)\in[-1+\eps,1-\eps]$ for every $s$, it immediately follows that \begin{equation*}
\lambda_s^1\to 1 \text{ and } \lambda_s^2,\lambda_s^3\to 0.
\end{equation*} From the above we conclude that $f_s^1\to f$ weakly. For every $s$, pick $z_s\in K$ such that $f_s^1=2\Phi(z_s)-\1$.  Then $\Phi(z_s)\to \Phi(z)$ weakly, and since $\Phi$ is a linear isometry, we get that $z_s\to z$ weakly. But $z$ is a point of continuity of $K$, so it follows that $\norm{z-z_s}\to 0$. Going back to $L_\infty[0,1]$, we get that $\nnorm{\Phi(z)-\Phi(z_s)}_\eps\to 0$, and thus $\nnorm{f-f_s}_\eps\to 0$, as we wanted. 
\end{proof} 

We will now prove that the set of Daugavet points of $B_\eps$ is weakly dense. The following approximation lemma will be useful throughout the rest of the article. 

\begin{lem}\label{lem:-1;1_weak_approximations_in_B}
    Let $\varphi\in B$ and let $(W_i)_{i\in I}$ be a collection of pairwise essentially disjoint measurable subsets of $[1/2,1]$ with non-zero measures. Then there exists a sequence $(\varphi_n)\subset B$ such that $\varphi_n\to \varphi$ weakly and such that the sets \begin{equation*}
    \{t\in W_i:\ \varphi_n(t)=1\} \text{ and } \{t\in W_i:\ \varphi_n(t)=-1\} 
    \end{equation*} have positive measure for every $i\in I$ and every $n\in\N$.
\end{lem}

\begin{proof}

For every $i\in I$, let $(B_i^n)$ and $(C_i^n)$ be sequences of measurable subsets of $W_i$ with positive measures such that any two distinct sets in these families are essentially disjoint (the existence of these subsets immediately follows from the fact that $\mu$ is $\sigma$-finite and purely non-atomic). Then set $b_i^n:=\1_{B_i^n}$ and $c_i^n:=\1_{C_i^n}$, and write $\varphi=(1-\eps)f+\eps g$ with $f\in B_{L_\infty[0,1]}$ and $g\in B_{\ker{(E_0)}}$. We define \begin{equation*}
    f_n:=f \cdot \left(\1-\sum_{i\in I}(b_i^n+c_i^n)\right)+\sum_{i\in I}(b_i^n-c_i^n) \text{ and } g_n:=g \cdot \left(\1-\sum_{i\in I}(b_i^n+c_i^n)\right)+\sum_{i\in I}(b_i^n-c_i^n).
\end{equation*} Note that these functions are well defined because the $B_i^n$ and $C_i^n$ are pairwise essentially disjoint. Furthermore, we have that $f_n=f$ and $g_n=g$ a.e. on $[0,1 ]\backslash\left(\bigcup_{i\in I} (B_i^n\cup C_i^n)\right)$, $\norm{f_n}_\infty, \norm{g_n}_\infty\leq 1$, and $f_n$ and $g_n$ are equal to $1$ a.e. on $B_i^n$ and to $-1$ a.e. on $C_i^n$. So $f_n\in B_{L_\infty[0,1]}$ and $g_n\in B_{\ker{(E_0)}}$, and calling to the weak convergence criteria (Theorem~\ref{thm:weakly_null_sequences_L_infty}), we get that $f_n\to f$ weakly and $g_n\to g$ weakly. Thus $\varphi_n=(1-\eps)f_n+\eps g_n$ does the job. 
\end{proof}

\begin{prop}\label{prop:weakly_dense_Daugavet_points}
Let $E$ be the set of all functions $u$ in $B_\eps$ of the form $u=\theta\lambda\psi+(1-\lambda)\varphi$, where $\theta,\lambda,\psi, \varphi$ satisfy the following conditions: \begin{enumerate}
    \item $\theta\in\{-1,1\}$ and $\lambda\in[0,1]$;

    \item $\psi=2h-\1$, where $h\in K_0$ is the image of a finitely supported element of $K$;

    \item $\varphi\in B$ and, for every $\alpha\in \mathbb N^{<\omega}$, the sets $\{t\in U_\alpha: \varphi(t)=1\}$ and $\{t\in U_\alpha: \varphi(t)=-1\}$ have positive measure.
\end{enumerate}
    Then $E$ is weakly dense in $B_\eps$. Furthermore, every $u\in E$ is a Daugavet point in $(L_\infty([0,1]),\nnorm{\cdot}_\eps)$.
\end{prop}

\begin{proof}
Let us first prove that the set $E$ is weakly dense in $B_\eps$. Since the sets $A$ and $B$ are convex, and since $\frac{A-A}{2}\subset B_{\ker(E_0)}\subset B$, it follows from \cite[Lemma~2.4]{blradv} that  \begin{equation*}
       \conv{(A\cup -A\cup B)}=\conv{(A\cup B)}\cup\conv{(-A\cup B)}.
\end{equation*} Hence, by symmetry, it is sufficient to prove that every element of $\conv{(A\cup B)}$ is the weak limit of a sequence in $E$. Let $u=\lambda \psi+(1-\lambda)\varphi$ with $\lambda\in[0,1]$, $\psi\in A$ and $\varphi\in B$. First, write $\psi=2h-\1$ with $h\in K_0$. Since the set of all finitely supported elements of $K$ is dense in $K$, and since $K$ is mapped isometrically onto $K_0$, we can find a sequence $(h_n)$ in $K_0$ which converges in norm to $h$ and such that every $h_n$ is the image of a finitely supported element of $K$.  We define $\psi_n:=2h_n-\1$. Next, for every $\alpha\in \N^{<\omega}$, pick a sequence $(W_\alpha^n)_{n\in \mathbb N}$ of pairwise disjoint non-empty open subsets of $A_\alpha$. By Lemma~\ref{lem:-1;1_weak_approximations_in_B}, we can find a sequence $(\varphi_n)$ in $B$ which converges weakly to $\varphi$ and such that the sets $\{t\in W_\alpha^n: \varphi_n(t)=1\}$ and $\{t\in W_\alpha^n: \varphi_n(t)=-1\}$ have positive measure for every $\alpha, n$. Clearly, $u_n=\lambda \psi_n+(1-\lambda)\varphi_n$ belongs to $E$ and converges weakly to $u$, so we are done.

Now let us prove that every element of $E$ is a Daugavet point. Since $E$ is symmetric, it is sufficient to show that every element of the form $u=\lambda\psi+(1-\lambda)\varphi$, where $\psi$ and $\varphi$ are as in the statement of the lemma, is a Daugavet point. So take such a function $u$ and, for every $\alpha\in \N^{<\omega}$, call $P_\alpha:=\{t\in A_\alpha: \varphi(t)=1\}$ and $N_\alpha:=\{t\in A_\alpha: \varphi(t)=-1\}$, which have positive measure. Note that by assumption, $\psi=-1$ a.e. on all but finitely many $A_\beta$, which means $u=-1$ a.e. on all but finitely many $N_\beta$. Let $S$ be a slice of $B_\eps$. Then $S$ has non-empty intersection with $A$, $-A$ or $B$, so there are three cases to consider. 

\textbf{Case~1.} If $A\cap S$ is non-empty, then, since $A=\cconv\{u_\alpha\}_{\alpha\in \N^{<\omega}}$, we have that $S$ must contain one of the functions $u_\alpha$. As previously observed, $u=-1$ a.e. on all but finitely many $N_\beta$, so since $u_{\alpha\smallfrown n}\to u_\alpha$ weakly, we can find $n_0\in \N$ such that $u_{\alpha\smallfrown n_0}\in S$ and $u$ is equal to $-1$ a.e. on $N_{\alpha\smallfrown n_0}$. But by definition, $u_{\alpha\smallfrown n_0}=1$ a.e. on $N_{\alpha\smallfrown n_0}\subseteq A_{\alpha\smallfrown n_0}$, so this implies that \begin{equation*}
       \{t\in [0,1]: \vert u_{\alpha\smallfrown n_0}(t)-u(t)\vert\geq 2\}
\end{equation*}has positive measure. Hence \begin{equation*}
\nnorm{u-u_{\alpha\smallfrown n_0}}_\eps\geq \norm{u-u_{\alpha\smallfrown n_0}}_\infty= 2.
\end{equation*}

\textbf{Case~2.} If $-A\cap S$ is non-empty, then $S$ must contain one of the $-u_\alpha$. Since those take the constant value $1$ on all except finitely many $A_\beta$, it suffices to compute the value of $u+u_\alpha$ on $N_\beta$, where $\beta$ is such that $u$ is equal to $-1$ a.e. on $U_\beta$ and $-u_\alpha$ is equal $1$ a.e. on $A_\beta\supset N_\beta$.

\textbf{Case~3.} Assume that $B\cap S$ is non-empty, and pick $\tilde{\varphi}$ in this set. By assumption, there exists $\alpha\in \N^{<\omega}$ such that $u=-1$ a.e. on the open set $N_\alpha$. Calling to Lemma~\ref{lem:-1;1_weak_approximations_in_B}, we can then approximate $\tilde{\varphi}$ by a sequence of elements of $B$ which are a.e. equal to $1$ on some subsets of positive measure of $N_\alpha$. Hence, without lost of generality, $\tilde{\varphi}$ satisfies this latter property, and it immediately follows that \begin{equation*}
 \nnorm{u-\tilde{\varphi}}_\eps\geq \Vert u-\tilde{\varphi}\Vert_\infty= 2. 
\end{equation*}
\end{proof}

As a direct consequence, we get:

\begin{cor}\label{cor:2K_0-1_Dpts}
Every function $u$ in the set $A$ is a Daugavet point in $(L_\infty [0,1],\nnorm{\cdot}_\eps)$. 
\end{cor}

\begin{proof}
    It immediately follows from Proposition~\ref{prop:weakly_dense_Daugavet_points} that $2\Phi(z)-\1$ is a Daugavet point in $(L_\infty [0,1],\nnorm{\cdot}_\eps)$ for every $z\in K$ with finite support. Since these elements are dense in $A=2K_0-\1$, and since the set of all Daugavet points is closed, we immediately get the desired conclusion. 
\end{proof}

\begin{rem}
Since $B_\eps$ admits non-empty relatively weakly open subsets of arbitrarily small diameter, it follows from Proposition~\ref{prop:weakly_dense_Daugavet_points} that some of the Daugavet points from the set $E$ are not super $\Delta$-points. Furthermore, it follows from Proposition~\ref{prop:PC_transfer} and Corollary~\ref{cor:2K_0-1_Dpts} that $B_\eps$ contains points which are simultaneously Daugavet points and points of continuity, so we get new examples in the line of  \cite{almt20}. 
\end{rem}

In particular, we get that the $u_\alpha$ are Daugavet points in $(L_\infty [0,1],\nnorm{\cdot}_\eps)$. But we can say a bit more about these functions. 

\begin{prop}\label{prop:u_alpha_diametral}
    Every $u_\alpha$ is a preserved extreme point of $B_\eps$. In particular, every $u_\alpha$ is simultaneously a Daugavet point and a ccw $\Delta$-point (in particular a super $\Delta$-point) in $(L_\infty [0,1],\nnorm{\cdot}_\eps)$.
\end{prop}

\begin{proof}
In order to prove that $u_\alpha$ is preserved extreme, it is sufficient to prove in virtue of Proposition \ref{prop:carapreext} that for any nets $(u_s)$ and $(v_s)$ in $\conv(A\cup -A\cup B)$ such that $\frac{u_s+v_s}{2}\to u_\alpha$ weakly,  we have that $u_s\to u_\alpha$ weakly. So pick such nets, and for every $s$, let us write \begin{equation*}
u_s=\lambda_s^1 f_s^1 - \lambda_s^2 f_s^2 + \lambda_s^3 f_s^3  \text{ and } v_s=\mu_s^1 g_s^1 - \mu_s^2 g_s^2 + \mu_s^3 g_s^3
\end{equation*} with $\lambda_s^i,\mu_s^i\in[0,1]$, $\sum_{i=1}^3\lambda_s^i=\sum_{i=1}^3\mu_s^i=1$, $f_s^1,f_s^2,g_s^1,g_s^2\in A$ and $f_s^3,g_s^3\in B$. Testing against $E_0$, we get: \begin{equation*}
E_0\left(\frac{u_s+v_s}{2}\right)=-\frac{\lambda_s^1+\mu_s^1}{2}+\frac{\lambda_s^2+\mu_s^2}{2}+\frac{\lambda_s^3}{2}E_0(f_s^3)+\frac{\mu_s^3}{2}E_0(g_s^3)\to E_0(u_\alpha)=-1.
\end{equation*} But since $E_0(f_s^3),E_0(g_s^3)>-1+\eps$, this implies that $\frac{\lambda_s^1+\mu_s^1}{2}\to 1$ and $\frac{\lambda_s^2+\mu_s^2}{2}, \frac{\lambda_s^3}{2}, \frac{\mu_s^3}{2}\to 0$.  Hence, $\lambda_s^1,\mu_s^1\to 1$,  $\lambda_s^2,\mu_s^2,\lambda_s^3,\mu_s^3\to 0$ and  $\frac{f_s^1+g_s^1}{2}\to u_\alpha=2\Phi(x_\alpha)-\1$ weakly. Write $f_s^1=2\Phi(y_s)-\1$ and $g_s^1=2\Phi(z_s)-\1$ with $y_s,z_s\in K$. Since $\Phi$ is isometric, we get, by the weak-to-weak continuity of bounded operators, that $\frac{y_s+z_s}{2}\to x_\alpha$ weakly in $K$. Since $x_\alpha$ is preserved extreme in $K$, it follows that $y_s\to x_\alpha$ weakly in $K$, and going back through $\Phi$, we conclude that $f_s^1\to u_\alpha$ weakly. Hence $u_s\to u_\alpha$ weakly, and we are done. The final part follows since it was observed in \cite[Proposition~3.13]{mpr23} that every preserved extreme $\Delta$-point is a ccw $\Delta$-point. 
\end{proof}

\begin{rem}\label{rem:superdeltanotsuperdauga}
    None of the $u_\alpha$ is super Daugavet.
    \end{rem}

\begin{proof}
For every $n\in\N$, let $u_\alpha^n:=2\left(h_\alpha+\frac{1}{n}\sum_{i=1}^nf_{\alpha\smallfrown i}\right)-\1$. We have \begin{equation*}
\nnorm{u_\alpha-u_\alpha^n}_\eps\leq \frac{1}{1-\eps}\norm{u_\alpha-u_\alpha^n}_\infty=\frac{1}{1-\eps}\norm{\Phi\left(\frac{1}{n}\sum_{i=1}^nx_{\alpha\smallfrown i}\right)}_\infty=\frac{1}{n(1-\eps)}.
\end{equation*} Consider 
\begin{equation*}
    \tilde W_{n,\rho}:= \left\{ f \in B_\varepsilon : E_0(f)<-1+\rho,\ \text{and}\ E_{\alpha \smallfrown i}(f) > 2\left(\frac{1}{n}-\rho \right)-1  \text{ for every }i\in \{1,...,n\} \right\}.
\end{equation*} Then we can show as in the proof of Proposition~\ref{prop:small_weakly_open_subsets_B_eps} that the diameter of  $\tilde W_{n, 1/n^2}$ goes to $0$ as $n$ goes to infinity. Since $u_\alpha^n\in \tilde W_{n, 1/n^2}$ and since the distance from $u_\alpha$ to $u_\alpha^n$ goes to 0, we get that $u_\alpha$ is not a super Daugavet point. 
\end{proof}

We will now show that $B_\eps$ contains points satisfying stronger diametral notions. We start by the following easy observation.

\begin{lem}\label{lem:infinite_-1,1_functions_belong_to_B}
Let $\varphi\in B_\eps$ be a function such that the sets
$$P_\alpha:=\{t\in U_\alpha: \varphi(t)=1\}\ \mbox{ and }\ N_\alpha:=\{t\in U_\alpha: \varphi(t)=-1\}$$
have positive measure for every $\alpha\in \mathbb N^{<\omega}$. Then $\varphi$ belongs to $B$.
\end{lem}

\begin{proof}
    Fix $\delta\in(0,1)$. We can find $h_1,h_2\in K_0$, $\psi\in B$ and $\lambda_1,\lambda_2,\lambda_3\geq 0$ such that $\lambda_1+\lambda_2+\lambda_3=1$ and $\norm{\varphi-(\lambda_1(2h_1-\1)+\lambda_2(-2h_2+\1)+\lambda_3\psi)}_\infty <\delta$. Also, we may assume without loss of generality that $h_1$ and $h_2$ are the images of finitely supported elements of $K$. So choose $\alpha\in \N^{<\omega}$ belonging in neither of these supports. Then $2h_1-\1=-1$ and $-2h_2+\1=1$  a.e. on $U_\alpha$. Taking into account that \begin{equation*}
        \esssup_{t\in P_\alpha}  (\lambda_1(2h_1-\1)+\lambda_2(-2h_2+\1)+\lambda_3\psi)(t) >1-\delta
    \end{equation*} we infer \begin{equation*}
        1-\delta<-\lambda_1+\lambda_2+\lambda_3\esssup_{t\in P_\alpha}  \psi(t) <-\lambda_1+\lambda_2+\lambda_3=1-2\lambda_1,
    \end{equation*} which implies $\lambda_1 < \delta/2$. Similarly we get \begin{equation*}
        -1+\delta>-\lambda_1+\lambda_2+\lambda_3\esssup_{t\in N_\alpha}\psi(t)>-\lambda_1+\lambda_2-\lambda_3=-1+2\lambda_2,
    \end{equation*} which implies $\lambda_2 < \delta/2$. Hence $\lambda_3>1-\delta$, and it follows that \begin{align*}
        \norm{\varphi-\psi}_\infty &\leq \lambda_1\norm{2h_1-\1}_\infty+\lambda_2\norm{2h_2-\1}_\infty+(1-\lambda_3)\norm{\psi}_\infty \\
        &\quad +\norm{\varphi-(\lambda_1(2h_1-\1)+\lambda_2(-2h_2+\1)+\lambda_3\psi)}_\infty < 3\delta.
    \end{align*} Since $B$ is closed, the conclusion follows.
\end{proof}

\begin{prop}\label{prop:infinite_-1,1_diametral}
Let $\varphi\in B_\eps$ be a function such that the sets
$$P_\alpha:=\{t\in U_\alpha: \varphi(t)=1\}\ \mbox{ and }\ N_\alpha:=\{t\in U_\alpha: \varphi(t)=-1\}$$
have positive measure for every $\alpha\in \mathbb N^{<\omega}$. Then $\varphi$ is simultaneously a super Daugavet point and a ccw $\Delta$-point in $(L_\infty[0,1],\nnorm{\cdot}_\eps)$.
\end{prop}

\begin{proof}
We first prove that $\varphi$ is a super Daugavet point. Let $W$ be a non-empty relatively weakly open subset of $B_\varepsilon$. By Proposition~\ref{prop:weakly_dense_Daugavet_points} there exists a function $u\in E$ which belongs to the weakly open set $W$. Write $u=\theta\lambda u_1+(1-\lambda)u_2$ with $\theta\in\{-1,1\}$, $\lambda\in[0,1]$, $u_1\in A$ and $u_2\in B$. We will assume that $\theta=1$, because the other case can be shown by the analogous method.  Up to approximating $u_2$ as in Lemma~\ref{lem:-1;1_weak_approximations_in_B}, we may assume without lost of generality that $u_2$ is equal to $-1$ a.e. on a subset of non-zero measure of $P_\alpha$ for every $\alpha\in\N^{<\omega}$.  Hence $u$ is equal to $-1$ a.e. on subsets of positive measure of all but finitely many $P_\alpha$, and it follows that $\nnorm{\varphi-u}_\eps\geq \norm{\varphi-u}_\infty=2$, as we wanted.

Next, let us prove that $\varphi$ is a ccw $\Delta$-point. Assume that $\varphi\in \sum_{i=1}^n\lambda_iW_i$, where the $W_i$ are non-empty relatively weakly open subset of $B_\eps$, $\lambda_i>0$, and $\sum_{i=1}^n\lambda_i=1$. Then write $\varphi=\sum_{i=1}^n\lambda_i\varphi_i$ with $\varphi_i\in W_i$. Since $\norm{\varphi_i}_\infty\leq 1$ for every $i$, it follows that \begin{equation*}
    \{t\in[0,1]:\ \varphi(t)=1\}=\bigcap_{i=1}^n \{t\in[0,1]:\ \varphi_i(t)=1\}
\end{equation*} and \begin{equation*}
    \{t\in[0,1]:\ \varphi(t)=-1\}=\bigcap_{i=1}^n \{t\in[0,1]:\ \varphi_i(t)=-1\}.
\end{equation*}Hence, every $\varphi_i$ is equal to $1$ a.e. on the $P_\alpha$ and to $-1$ a.e. on the $N_\alpha$. In particular, every $\varphi_i$ belongs to $B$ by Lemma~\ref{lem:infinite_-1,1_functions_belong_to_B}.

By Lemma~\ref{lem:-1;1_weak_approximations_in_B}, we can find for every $i$ a sequence of functions $(\varphi_i^k)$ in $B_\eps$ which converges weakly to $\varphi_i$ and such that $\varphi_i^k=-1$ a.e. on some subsets $A_\alpha\subseteq P_\alpha$ of positive measure for every $\alpha\in \mathbb N^{<\omega}$ and every $i,k$. Then we get that for large enough $k$'s, each $\varphi_i^k$ belongs to the corresponding $W_i$. It follows that $\sum_{i=1}^n\lambda_i\varphi_i^k\in \sum_{i=1}^n\lambda_iW_i$. Finally, \begin{equation*}
    \nnorm{\varphi-\sum_{i=1}^n\lambda_i\varphi_i^k}_\eps\geq  \norm{\varphi-\sum_{i=1}^n\lambda_i\varphi_i^k}_\infty \geq \esssup_{t\in A_\alpha} \abs{\varphi(t)-\sum_{i=1}^n\lambda_i\varphi_i^k(t)}=2,
\end{equation*} so the conclusion follows.
\end{proof}

\begin{cor}\label{cor:super_Delta_norming}

The set of all points of $B_\eps$ which are simultaneously Daugavet points and ccw $\Delta$-points is norming.

\end{cor}

\begin{proof}

Every slice $S$ of $B_\eps$ intersects either $A$, $-A$ or $B$.  But since every slice of $A$ must contain one of the $u_\alpha$, and since functions taking value $1$ and $-1$ on some subsets of positive measure of every $U_\alpha$ are weakly dense in $B$, the result immediately follows from Proposition~\ref{prop:u_alpha_diametral} and \ref{prop:infinite_-1,1_diametral}.
\end{proof}

Finally, let us prove that, for small enough $\eps$'s, there are points in $B_\varepsilon$ which are not $\Delta$-points.

\begin{proposition}
    If $\varepsilon<\frac{1}{7}$, then the space $\left(L_\infty([0,1]),\nnorm{\cdot}_\varepsilon\right)$ fails the DLD2P.
\end{proposition}
\begin{proof}

Consider the function $f:=(1-\varepsilon)\left(\1_{[0,\frac{1}{4}]}-\1_{[\frac{1}{4},\frac{1}{2}]}\right)$, which belongs to $(1-\varepsilon)B_{L_{\infty}([0,1])}$, and thus to $B_{\varepsilon}$. We will show that $\tilde{f}=\frac{f}{\nnorm{f}_\varepsilon}$ is not a $\Delta$-point. Consider the functional $F$ on $L_{\infty}[0,1]$ given by
\begin{equation*}
F(\varphi)=4\cdot\left(\int_0^{\frac{1}{4}}\varphi(t)dt-\int_{\frac{1}{4}}^{\frac{1}{2}}\varphi(t) dt\right)
\end{equation*} for every $\varphi\in L_{\infty}[0,1]$. Since $F(\varphi)=0$ for every $\varphi\in A$, we obtain the following estimate on the norm of $F$:
\begin{equation*}
    \nnorm{F}_\varepsilon=\sup\{F(\varphi)\colon \varphi\in A\cup(-A)\cup B\}\leq 2-\varepsilon.
\end{equation*}
This in turn allows us to provide a lower bound on the norm of $f$:
\begin{equation*}
    \nnorm{f}_\varepsilon\geq \frac{F(f)}{2-\varepsilon}=\frac{2(1-\varepsilon)}{2-\varepsilon}.
\end{equation*} We now define the slice of $B_\varepsilon$ which will witness that $\tilde{f}$ is not a $\Delta$-point. Let $\eta>0$ such that $3\eta<1-4\varepsilon$ and $\frac{5}{2}\varepsilon+\eta<\frac{1}{2}$, and consider
\begin{equation*}
    S := \left\{\varphi\in B_{\varepsilon}\colon F(\varphi)> 2(1-\varepsilon)-\eta\right\}.
\end{equation*} The set $S$ is nonempty, since $f$, and consequently $\tilde{f}$, belong to $S$, so it does indeed define a slice of $B_\eps$. To finish the proof, it suffices to uniformly bound away from $2$ the distance of any function in $S$ to $\tilde{f}$. In fact, since $B_\varepsilon$ is the closed convex hull of $A\cup -A\cup B$, and since no function in $A$ or $-A$ belongs to $S$, it is enough to uniformly bound away from $2$ the distance of any function in $S\cap B$ to $\tilde{f}$ (calling to \cite[Lemma 2.1]{jr22}). So fix $\varphi\in B$ such that $F(\varphi)>2(1-\varepsilon)-\eta$. Since $\|\varphi\|_\infty\leq 1$, we obtain that
\begin{align*}
    2(1-\varepsilon)-\eta&<F(\varphi)=E_0(\varphi)-4\cdot\int_{1/4}^{1/2}\varphi(t)dt\\
    &\leq E_0(\varphi)+1.
\end{align*}
With the previous equation and with the fact that $E_0(\varphi) \leq (1-\varepsilon)$, we get that $E_0(\varphi)$ belongs to $[1-2\varepsilon-\eta,1-\varepsilon]$. On the other hand, we have that
\begin{equation*}
    1-\varepsilon\leq E_0(\tilde{f}) = \frac{1-\varepsilon}{\nnorm{f}_\varepsilon}\leq\frac{2-\varepsilon}{2}.
\end{equation*}
Defining $\tau = E_0(\tilde{f}-\varphi)$, we obtain, combining both estimates, that:
\begin{equation}
\label{estimate_on_tau}
0\leq \tau \leq \frac{2-\varepsilon}{2}-(1-2\varepsilon-\eta)=\frac{3}{2}\varepsilon+\eta.
\end{equation}
Now, consider the function $h=\tilde{f}-\varphi-\tau\cdot\1_{[0,\frac{1}{4}]}$, which clearly satisfies $E_0(h)=0$. Since the inequality $-\varepsilon\leq\tilde{f}-\varphi\leq 2-\frac{\varepsilon}{2}$ holds almost everywhere on $[0,\frac{1}{4}]$, by the choice of $\eta$ and equation \eqref{estimate_on_tau} we get that $\|h\|_\infty\leq  2-\frac{\varepsilon}{2}-\tau$. Next, observe that $\frac{h}{\|h\|_\infty}$ belongs to $B$, since $\frac{h}{\|h\|_\infty}$ is contained in both $B_{L_{\infty}[0,1]}$ and $B_{\ker(E_0)}$, and can be trivially written as $\frac{h}{\|h\|_\infty}=(1-\varepsilon)\frac{h}{\|h\|_\infty}+\varepsilon \frac{h}{\|h\|_\infty}$. In other words, we have that $\nnorm{h}_\varepsilon=\|h\|_\infty$. On the one hand, this immediately yields the estimate
\begin{equation*}
    \nnorm{h}_\varepsilon=\|h\|_\infty \leq 2-\frac{\varepsilon}{2}-\tau.
\end{equation*}
On the other hand, using the triangle inequality, and the fact that $\nnorm{\cdot}_\varepsilon \leq \frac{1}{1-\varepsilon}\|\cdot\|_\infty$,  we obtain
\begin{equation*}
    \nnorm{h}_\varepsilon\geq \nnorm{\tilde{f}-\varphi}_\varepsilon-\nnorm{\tau\cdot\1_{[0,\frac{1}{4}]}}_\varepsilon\geq \nnorm{\tilde{f}-\varphi}_\varepsilon-\frac{\tau}{1-\varepsilon}.
\end{equation*} Combining the two previous bounds, doing a few more simple computations, and using equation \eqref{estimate_on_tau}, we get that
\begin{equation*}
    \nnorm{\tilde{f}-\varphi}_\varepsilon\leq 2-\frac{\varepsilon}{2}-\tau+\frac{\tau}{1-\varepsilon}=2-\frac{\varepsilon}{2}+\frac{\tau\varepsilon}{1-\varepsilon}\leq 2-\frac{\varepsilon}{2}+\frac{\varepsilon(\frac{3}{2}\varepsilon+\eta      )}{1-\varepsilon}=2-\varepsilon\left(\frac{1}{2}-\frac{\frac{3}{2}\varepsilon+\eta}{1-\varepsilon}\right),
\end{equation*}
and the last term is smaller than or equal to $2-\frac{\varepsilon}{4}$. Indeed \begin{equation*}
    \begin{split}2-\varepsilon\left(\frac{1}{2}-\frac{\frac{3}{2}\varepsilon+\eta}{1-\varepsilon}\right)\leq 2-\frac{\varepsilon}{4}& \Leftrightarrow \frac{1}{4}\leq \frac{1}{2}-\frac{\frac{3}{2}\varepsilon+\eta }{1-\varepsilon}\Leftrightarrow \frac{3}{2}\varepsilon+\eta\leq \frac{1}{4}(1-\varepsilon)\\
& \Leftrightarrow 6\varepsilon+4\eta\leq 1-\varepsilon\Leftrightarrow 7\varepsilon+4\eta\leq 1,
\end{split}
\end{equation*} and the last inequality holds by the choice of $\eta$.
\end{proof}

%\section{Lattice embedding}

%\begin{lemma}\label{lemma:pairwisedisjoint}
%Let $(f_n)$ be a sequence in $S_{L_\infty([0,1])}$ such that $\supp(f_n)$ is pairwise disjoint. Then the operator $T:c_0\longrightarrow L_\infty([0,1])$ given by $T(e_n)=f_n \forall n\in\mathbb N$ is a linear isometry.

%In particular, if $f_n\geq 0$, the above defines a lattice isometric embedding.
%\end{lemma}

%\begin{proof}
%Let $\lambda_1,\ldots, \lambda_n\in\mathbb R$. Let $g:=\sum_{i=1}^n \lambda_i f_i$, and let us prove that $\Vert g\Vert=\max\limits_{1\leq i\leq n}\vert \lambda_i\vert$. In view of the lattice order defined in $L_\infty([0,1])$ it is clear that 
%$$\vert g\vert=\sum_{i=1}^n \vert \lambda_i f_i\vert=\bigvee_{i=1}^n \vert \lambda_i f_i\vert.$$
%Since $\vert \lambda_i f_i\vert$ are pairwise disjoint elements in $L_\infty([0,1])$ and it is an AM-space, we infer that
%$$\Vert g\Vert=\Vert \vert g\vert \Vert=\left\Vert \bigvee_{i=1}^n \vert \lambda_i f_i\vert \right\Vert=\max\limits_{1\leq i\leq n} \Vert \lambda_i f_i\Vert=\max\limits_{i=1}^n \vert \lambda_i\vert,$$
%as desired.
%\end{proof}

\section*{Acknowledgements}

We thank the organisers of 2023 ICMAT-IMAG Doc-Course in Functional Analysis for the support and the hospitality during the development of the supervised research program that resulted in this paper.

This work was supported by MCIN/AEI/10.13039/501100011033/FEDER, UE: grant
\\ PID2021-122126NB-C31 (L\'opez-P\'erez, Mart\'in, Quero and Rueda Zoca), grant PID2021-122126NB-C33 (Cobollo and Quilis), and grants PID2019-105011GB-I00 and PID2022-139449NB-I00 (Cobollo); Junta de Andaluc\'ia: Grants FQM-0185 (L\'opez-P\'erez, Mart\'in, Quero and Rueda Zoca); Grant PGC2018-097286-B-I00 by MCIU/AEI/FEDER, UE (Rodríguez-Vidanes).

The research of Ch. Cobollo was also supported by Generalitat Valenciana (through Project PROMETEU/2021/070 and the predoctoral contract CIACIF/2021/378), and by Universitat Polit\`ecnica de Val\`encia.

The research of G. L\'opez-P\'erez and M. Mart\'in was also supported by  MICINN (Spain) Grant
CEX2020-001105-M (MCIU, AEI).

The work of Y. Perreau was supported by the Estonian Research Council grant SJD58.

The research of A. Quero was also supported by the Spanish Ministerio de Universidades through a predoctoral contract FPU18/03057.

The research of A. Quilis was also supported by the French ANR project No. ANR-20-CE40-0006.

The research of D.L. Rodríguez-Vidanes was also supported by MCIU and the European Social Fund through a ``Contrato Predoctoral para la Formación de Doctores, 2019'' (PRE2019-089135) and by the ``Instituto de Matemática Interdisciplinar'' (IMI).

The research of A. Rueda Zoca was also funded by Fundaci\'on S\'eneca: ACyT Regi\'on de Murcia grant 21955/PI/22 and by Generalitat Valenciana project CIGE/2022/97.


\begin{thebibliography}{XXX}

\bibitem {aalmppv23} T. A. Abrahamsen, R. J. Aliaga, V. Lima, A. Martiny, Y. Perreau, A. Prochazka and T. Veeorg, \textit{A relative version of Daugavet-points and the Daugavet property}, preprint. Available at ArXiV.org with reference \href{https://arxiv.org/abs/2306.05536}{arXiv:2306.05536}.

\bibitem {aalmppvfirst23} T. A. Abrahamsen, R. J. Aliaga, V. Lima, A. Martiny, Y. Perreau, A. Prochazka and T. Veeorg, \textit{Delta-points and their implications for the geometry of Banach spaces}, preprint. Available at ArXiV.org with reference \href{https://arxiv.org/abs/2303.00511}{arXiv:2303.00511}.

\bibitem {ahntt} T. A. Abrahamsen, P. H\'{a}jek, O. Nygaard, J. Talponen, and S. Troyanski, \emph{Diameter 2 properties and convexity}, Studia Math. \textbf{232} (2016), no.~3, 227--242. \MR{3499106}

\bibitem {ahlp20}  	T. A. Abrahamsen, R. Haller, V. Lima, and K. Pirk, \textit{Delta- and Daugavet-points in Banach spaces}, Proc. Edinb. Math. Soc.  \textbf{63} (2) (2020), 475--496.

\bibitem{almp22} T. A. Abrahamsen, V. Lima, A. Martiny, and Y. Perreau, \textit{Asymptotic geometry and {D}elta-points}, Banach J. Math. Anal., 16 (2022), p.~Paper No. 57.
  
\bibitem {almt20} T. A. Abrahamsen, V. Lima, A. Martiny and S. Troyanski, \textit{Daugavet- and delta-points in Banach spaces with unconditional bases}, Trans. Amer. Math. Soc. Ser. B \textbf{8} (2021), 379--398.

\bibitem {abal} Y. A. Abramovich, C. D. Aliprantis, \textit{An Invitation to Operator Theory}, American Mathematical Society, Rhode Island (2002).

\bibitem{aorlecture} S. Argyros, E. Odell, H. Rosenthal, (1988). \textit{On certain convex subsets of $c_0$.} In: E. W. Odell, H. P. Rosenthal, (eds) Functional Analysis. Lecture Notes in Mathematics, \textbf{1332}. Springer, Berlin, Heidelberg. \href{https://doi.org/10.1007/BFb0081613}{https://doi.org/10.1007/BFb0081613}.

\bibitem {blrjmaa} J. Becerra Guerrero, G. L\'opez-P\'erez and A. Rueda Zoca, \textit{Big slices versus big relatively weakly open subsets in Banach spaces}, J. Math. Anal. Appl. \textbf{428} (2015), 855--865.

\bibitem {blradv} J. Becerra Guerrero, G. L\'opez-P\'erez and A. Rueda Zoca, \textit{Extreme differences between weakly open subsets and convex combination of slices in Banach spaces}, Adv. Math. \textbf{269} (2015), 56--70.

\bibitem{blr18} J.\ Becerra Guerrero, G.\ L\'opez-P\'erez, and A.\ Rueda Zoca, \textit{Diametral diameter two properties in Banach spaces}, J. Conv. Anal. \textbf{25}, 3 (2018), 817--840.

\bibitem {cj23} G. Choi, M. Jung, \textit{The Daugavet and Delta-constants of points in Banach spaces}, preprint. Available at ArXiV.org with reference \href{https://arxiv.org/abs/2307.10647}{arXiv:2307.10647}.

\bibitem {fab} M. Fabian, P. Habala, P. H\'ajek, V. Montesinos, J. Pelant, and V. Zizler, \textit{Functional Analysis and Infinite dimensional Geometry}, CMS Books in Mathematics, Springer-Verlag, New York, 2001.

\bibitem {ggms} N. Ghoussoub, G. Godefroy, B. Maurey, W. Schachermayer, \textit{Some topological and geometrical structures in Banach spaces}, Mem. Amer. Math. Soc. \textbf{387} (1987), 116 p.

\bibitem {hlpv23} R. Haller, J. Langemets, Y. Perreau, and T. Veeorg, \textit{Unconditional bases and daugavet renormings}, preprint. Available at ArXiV.org with reference \href{https://arxiv.org/abs/2303.07037}{arXiv:2303.07037}.
  
\bibitem {jr22} M. Jung and A. Rueda Zoca, \textit{Daugavet points and $\Delta$-points in Lipschitz-free spaces}, Studia Math. \textbf{265} (1) (2022), 37--55.

\bibitem{kadets20} V.\ Kadets, \textit{The diametral strong diameter 2 property of Banach spaces is the same as the Daugavet property}, Proc. Amer. Math. Soc. \textbf{149} (2021), 2579--2582.

\bibitem {kssw} V. Kadets, R. Shvidkoy, G. Sirotkin and D. Werner, \textit{Banach spaces with the Daugavet property}, Trans. Am. Math. Soc. \textbf{352}, No.2 (2000), 855--873.

\bibitem {klt23} A. Kami\'nska, H. J. Lee and H. J. Tag, \textit{Daugavet and diameter two properties in Orlicz-Lorentz spaces}, J. Math. Anal. Appl. \textbf{529}, 2 (2024), article 127289.

\bibitem {luiscathesis} L. C. Garc\'ia~Lirola,\textit{Convexity, optimization and geometry of the ball in Banach
  spaces}, PhD thesis, Universidad de Murcia, 2017.
\newblock Available at \emph{DigitUM} with reference
  \href{http://hdl.handle.net/10201/56573}{http://hdl.handle.net/10201/56573}.

\bibitem {mpr23} M. Mart\'in, Y. Perreau and A. Rueda Zoca, \textit{Diametral notions for elements of the unit ball of a Banach space}, preprint. Available at ArXiV.org with reference \href{https://arxiv.org/abs/2301.04433}{arXiv:2301.04433}

\bibitem {shv00} R. V. Shvidkoy, \textit{Geometric aspects of the Daugavet property}, J. Funct. Anal. \textbf{176} (2000), 198-212.

\bibitem {toland20} J. Toland, \emph{The dual of {$L_{\infty}(X,\mathcal{L},\lambda)$}, finitely
additive measures and weak convergence---a primer}, SpringerBriefs in Mathematics (2020).
              
\bibitem {veeorg} T. Veeorg, \textit{Characterizations of Daugavet points and delta-points in Lipschitz-free spaces}, Studia Math. \textbf{268}, 2 (2023), 213--233.

\bibitem {wernersurvey01} D. Werner, \textit{Recent progress on the Daugavet property}, Irish Math. Soc. Bulletin \textbf{46} (2001), 77--97.

\end{thebibliography}
\end{document}